\documentclass[12pt]{article}
\usepackage{amsmath,amsthm, mathtools}
\everymath{\scriptsize}
\everydisplay{\scriptsize}
\makeatother
\usepackage{graphicx}
\usepackage{float}
\usepackage{color}
\usepackage{fancyhdr}
\usepackage{verbatim}
\usepackage{amsfonts,longtable}
\usepackage[margin=1in]{geometry}
\usepackage[numbers,sort&compress]{natbib}

\def\bc{\begin{center}}
\def\ec{\end{center}}
\newtheorem{definition}{Definition}[section]
\newtheorem{theorem}{Theorem}[section]

\newcommand{\pd}[2]{\frac{\partial #1}{\partial #2}}

\begin{document}
\bc {\bf Analytical Methods for Solving Non-Autonomous Higher-Order Nonlinear Difference Equations
  }\ec
\medskip
\bc
Tshifhiwa Murovhi$^{1}$, Mensah Folly-Gbetoula $^{1}$\footnote{Corresponding author:\\ Mensah.Folly-Gbetoula@wits.ac.za(M. Folly-Gbetoula)\\
	ORCID: 0000-0002-3046-0679}, Kwassi Anani $^{2}$
\vspace{1cm}
\\$^{1}$School of Mathematics, University of the Witwatersrand, Wits 2050, Johannesburg,\\ South Africa.\\
$^{2}$ Department of Mathematics, University of Lomé,
{Lomé},
	{01 BP 1515}, {Togo}

\ec
\begin{abstract}
\noindent
A system of four-dimensional fifth-order rational difference equations with variable coefficients is considered. The method employed is the well-known Lie symmetry analysis initially developed for differential equations. Linearised symmetry conditions are applied, reduction operators are
introduced, functional equations are solved, and symmetry generators are derived. The admitted point symmetries are then used to construct closed-form solutions. This work is a non-autonomous generalisation of autonomous fifth-order rational systems in the literature.
\end{abstract}
\section{Introduction and background} \setcounter{equation}{0}
Difference equations are the discrete analogue of differential equations. They are used in mathematical biology, population dynamics, finance, numerical analysis, control theory and other
areas in which changes occur step by step. In general, they are difficult, especially when they are of higher-order. Lie symmetry analysis provides a systematic way of reducing such problems. The main idea is to find transformations that preserve the form of the system. These transformations generate canonical coordinates and invariant quantities, which can then be used to reduce the system.
\noindent The present paper is motivated by existing work on symmetry methods for difference equations. Hydon developed a systematic framework for Lie symmetries of difference equations \cite{Hydon2014}, extending earlier ideas of Maeda \cite{Maeda1987}. More recent developments include \cite{GocenMensah2024, MensahAnani2024, M2,M3,M4},
who used Lie symmetry methods to obtain analytical solutions
for higher-order nonlinear difference equations.

\noindent This work studies the system {\scriptsize
\begin{equation}
	\begin{aligned}
		x_{n+1}&=\frac{x_{n-3}y_{n-4}}
		{y_n\left(a_n+b_nx_{n-3}y_{n-4}w_{n-1}z_{n-2}\right)},\\[6pt]
		y_{n+1}&=\frac{y_{n-3}w_{n-4}}
		{w_n\left(c_n+d_ny_{n-3}w_{n-4}z_{n-1}x_{n-2}\right)},\\[6pt]
		w_{n+1}&=\frac{w_{n-3}z_{n-4}}
		{z_n\left(e_n+f_nw_{n-3}z_{n-4}x_{n-1}y_{n-2}\right)},\\[6pt]
		z_{n+1}&=\frac{z_{n-3}x_{n-4}}
		{x_n\left(g_n+h_nz_{n-3}x_{n-4}y_{n-1}w_{n-2}\right)},
	\end{aligned}
	\label{eq:main_system}
\end{equation}}
where the
initial values are chosen so that every denominator occurring in the system remains non-zero. The autonomous fifth-order sign systems studied by Elsayed and Alofi \cite{ElsayedAlofi2023}
provide the constant-coefficient systems that are recovered as special cases of the present work. For related work, see \cite{N1,M1}. The principal contributions of this paper are summarised as follows: the derivation of Lie point symmetries for the four-dimensional
	fifth-order non-autonomous system; the detailed construction of canonical coordinates and product
	invariants; the reduction of the system using the invariant quantities; the extension of the reduction through reciprocal invariant variables; the derivation of closed product formulas; the recovery of autonomous sign systems as special cases; the inclusion of a numerical illustration of the periodicity result.
\subsection{Preliminaries}
Let {$S$} be the forward shift operator defined as { $S^j(u_n)=u_{n+j}$}. A fifth-order system may be written in the form
\begin{equation}
	S^5(u_k)=\Omega_k(n,[u_k]), \qquad k=1,2,3,4,
	\label{eq:abstract_system}
\end{equation}
where $[u_k]$ denotes the dependent variables together with their required
shifts, and $\Omega_k$ are some smooth functions.
\begin{definition}
	A one-parameter local group of point transformations
	\begin{align}
	(n,u_k)\mapsto \left(n,\widetilde u_k=u_k+\varepsilon Q_k(n,u_k)+O(\varepsilon^2)\right)
	\end{align}
	is called a symmetry of \eqref{eq:abstract_system} if it maps each solution
	to another solution of the same system.
\end{definition}
\noindent The functions $Q_k(n,u_k)$ are called the characteristics of the group of transformations and $\varepsilon$ is the group parameter. The 
infinitesimal generator is
\begin{equation}
	X=\sum_{k=1}^4Q_k(n,u_k)\frac{\partial}{\partial u_k}.
\end{equation}
The fourth prolongation is
\begin{equation}
	X^{[4]}=\sum_{j=0}^4\sum_{k=1}^4S^jQ_k
	\frac{\partial}{\partial S^j(u_k)}.
\end{equation}
The linearised symmetry condition is
\begin{equation}
	S^5Q_k-X^{[4]}(\Omega_k)=0,
	\qquad k=1,2,3,4.
	\label{eq:LSC}
\end{equation}
\begin{definition}
	A smooth function $\zeta$ is an invariant function if
	\begin{equation}
		X(\zeta)=0.
	\end{equation}
\end{definition}
\noindent Canonical coordinates associated with the characteristics are given by \cite{NV}
\begin{equation}
	s_k=\int \frac{du_k}{Q_k(n,u_k)}, \qquad k=1,2,3,4.
	\label{eq:canonical}
\end{equation}
To begin, we replace $n$ by $n+4$ in \eqref{eq:main_system}, and we write
$A_n=a_{n+4}, B_n=b_{n+4}, C_n=c_{n+4}, D_n=d_{n+4}, E_n=e_{n+4},\; F_n=f_{n+4},\; G_n=g_{n+4},\; H_n=h_{n+4}.$
This yields the fully shifted form of \eqref{eq:main_system}:
\begin{equation}
	\begin{aligned}
		x_{n+5}=\Omega_1
		&=\frac{x_{n+1}y_n}
		{y_{n+4}\left(A_n+B_nx_{n+1}y_nw_{n+3}z_{n+2}\right)},\\
		y_{n+5}=\Omega_2
		&=\frac{y_{n+1}w_n}
		{w_{n+4}\left(C_n+D_ny_{n+1}w_nz_{n+3}x_{n+2}\right)},\\
		w_{n+5}=\Omega_3
		&=\frac{w_{n+1}z_n}
		{z_{n+4}\left(E_n+F_nw_{n+1}z_nx_{n+3}y_{n+2}\right)},\\
		z_{n+5}=\Omega_4
		&=\frac{z_{n+1}x_n}
		{x_{n+4}\left(G_n+H_nz_{n+1}x_ny_{n+3}w_{n+2}\right)}.
	\end{aligned}
	\label{eq:forward_shifted_full}
\end{equation}
We seek point symmetries with characteristics
$Q_1=Q_1(n,x_n),\; Q_2=Q_2(n,y_n),\; Q_3=Q_3(n,w_n)$ and $Q_4=Q_4(n,z_n).$ The aim is to apply the symmetry condition, reduce and then solve for the characteristics. The linearised symmetry condition for the first equation reads
\begin{align}
	S^5Q_1
	&-S^1Q_1\pd{\Omega_1}{x_{n+1}}
	-Q_2\pd{\Omega_1}{y_n}
	-S^3Q_3\pd{\Omega_1}{w_{n+3}}
	-S^2Q_4\pd{\Omega_1}{z_{n+2}}
	-S^4Q_2\pd{\Omega_1}{y_{n+4}}=0
\end{align}
simplifies to
\begin{align}
	S^5Q_1
	&+\frac{B_nx_{n+1}^2y_n^2\left(z_{n+2}S^3Q_3+w_{n+3}S^2Q_4\right)-A_n\left(y_nS^1Q_1+x_{n+1}Q_2\right)}{y_{n+4}(A_n+B_nx_{n+1}y_nw_{n+3}z_{n+2})^2}
	+\frac{x_{n+1}y_nS^4Q_2}{y_{n+4}^2(A_n+B_nx_{n+1}y_nw_{n+3}z_{n+2})}=0.
	\label{eq:LSC1full}
\end{align}
Applying the operator $\mathcal O_1=\frac{\partial}{\partial x_{n+1}}-\frac{y_n}{x_{n+1}}\frac{\partial}{\partial y_n}$ to \eqref{eq:LSC1full}, the cross terms cancel and
the surviving terms reduce to
\begin{equation}
	(S^1Q_1)'-Q_2'+\frac{Q_2}{y_n}-\frac{S^1Q_1}{x_{n+1}}=0.
	\label{eq:red1}
\end{equation}
\noindent The remaining operators
\begin{align}
	\mathcal O_2=\frac{\partial}{\partial y_{n+1}}-\frac{w_n}{y_{n+1}}\frac{\partial}{\partial w_n},
	\mathcal O_3=\frac{\partial}{\partial w_{n+1}}-\frac{z_n}{w_{n+1}}\frac{\partial}{\partial z_n},
	\mathcal O_4&=\frac{\partial}{\partial z_{n+1}}-\frac{x_n}{z_{n+1}}\frac{\partial}{\partial x_n}
\end{align}
 are obtained in the same way by
exploiting the cyclic structure of the original system. Each operator
is constructed to eliminate the denominator terms appearing in the
corresponding linearised symmetry condition, thereby reducing the
problem to equations involving fewer shifts.
\noindent Applying these operators gives
\begin{align}
	(S^1Q_2)'-Q_3'+\frac{Q_3}{w_n}-\frac{S^1Q_2}{y_{n+1}}&=0,\label{eq:red2}\\
	(S^1Q_3)'-Q_4'+\frac{Q_4}{z_n}-\frac{S^1Q_3}{w_{n+1}}&=0,\label{eq:red3}\\
	(S^1Q_4)'-Q_1'+\frac{Q_1}{x_n}-\frac{S^1Q_4}{z_{n+1}}&=0.
	\label{eq:red4}
\end{align}

\noindent Differentiating \eqref{eq:red1} with respect to $y_n$ gives
\begin{align}
-Q_2''+\frac1{y_n}Q_2'-\frac1{y_n^2}Q_2=0.
\end{align}
The same argument applies to all four characteristics, so
\begin{equation}
	Q_k''-\frac1{u_k}Q_k'+\frac1{u_k^2}Q_k=0,
	\qquad k=1,2,3,4.
	\label{eq:Qode}
\end{equation}
Setting $Q_k=u_k^m$ gives the indicial equation $(m-1)^2=0.$ 
Hence $m=1$ is a repeated root and the general solution of
\eqref{eq:Qode} is
\begin{align}
Q_k(n,u_k)
=
\bigl(\beta_k(n)+B_k(n)\ln|u_k|\bigr)u_k.
\end{align}
We restrict attention to the
subclass $B_k(n)=0$, which is sufficient to construct the canonical coordinates, invariants and invariant reductions developed below. Therefore,

\begin{equation}
	Q_k(n,u_k)=\beta_k(n)u_k,
	\qquad k=1,2,3,4.
	\label{eq:Qbeta}
\end{equation}
We substitute the characteristics and their shifts in the symmetry conditions. Separating coefficients in the resulting equations gives
\begin{align}
&	\beta_k(n)-\beta_k(n+4)=0, \quad k=1,2,3,4,\label{eq:coefC} \\
&\beta_1(n+1)+\beta_2(n)+\beta_3(n+3)+\beta_4(n+2)=0,\label{eq:coefA}  \\
&\beta_1(n+2)+\beta_2(n+1)+\beta_3(n)+\beta_4(n+3)=0,\label{eq:coefB}  \\
&\beta_1(n+3)+\beta_2(n+2)+\beta_3(n+1)+\beta_4(n)=0,\label{eq:coefD}  \\
&\beta_1(n)+\beta_2(n+3)+\beta_3(n+2)+\beta_4(n+1)=0,\label{eq:coefE} 
\end{align}
Equations \eqref{eq:coefC}-\eqref{eq:coefE} impose additional compatibility conditions on the coefficient functions $\beta_i$, $i=1,2,3,4$. ? It follows that 
\begin{align}
	\beta_1(n)=&m_1+m_2 i^n+m_3(-1)^n+m_4(-i)^n\\
	\beta_2(n)=&m_5+m_6 i^n+m_7(-1)^n+m_8(-i)^n\\
	\beta_3(n)=&m_9+m_{10} i^n+m_{11}(-1)^n+m_{12}(-i)^n\\
	\beta_4(n)=&-m_1-m_5-m_9+(m_6-im_2-im_{10})i^n+(m_3-m_7+m_{11})(-1)^n+(-im_4+m_8+im_{12})(-i)^n,
\end{align}
for some constants $ m_i,\; i=1,\dots, 12$.
Thus, the symmetries are given by{\scriptsize
\begin{align}
	X_1=&x_n\partial_{x_n}-z_n\partial_{z_n},\;
		X_{2}=i^nx_n\partial_{x_n}-i^{n+1} z_n\partial_{z_n},\;
		X_{3}=(-1)^nx_n\partial_{x_n}
+(-1)^nz_n\partial_{z_n},\;
		X_{4}=(-i)^nx_n\partial_{x_n}+(-i)^{n+1}z_n\partial_{z_n}\nonumber\\
		X_{5}=&y_n\partial_{y_n}
	-z_n\partial_{z_n},\, 	X_{6}=i^ny_n\partial_{y_n}
	+i^nz_n\partial_{z_n} , \;  
		X_{7}=(-1)^ny_n\partial_{y_n}+(-1)^{n+1} z_n\partial_{z_n},\;
		X_{8}=(-i)^n y_n\partial_{y_n}
	+(-i)^nz_n\partial_{z_n}\nonumber \\
		X_{9}=& w_n\partial_{w_n}-z_n\partial_{z_n},\;
		X_{10}=i^n w_n\partial_{w_n}-i^{n+1}z_n\partial_{z_n},\;
		X_{11}=(-1)^nw_n\partial_{w_n}+(-1)^nz_n\partial_{z_n},\nonumber\\
		X_{12}=&(-i)^nw_n\partial_{w_n}-(-i)^{n+1}z_n\partial_{z_n}.
\end{align}}


The canonical coordinates are
\begin{align}
	s_1(n)&=\int\frac{dx_n}{Q_1(n,x_n)}=\frac{1}{\beta_1(n)} \ln|x_n|,\qquad s_2(n)=\frac{1}{\beta_2(n)}\ln|y_n|, \qquad 
	s_3(n)=\frac{1}{\beta_3(n)}\ln|w_n|,\qquad s_4(n)=\frac{1}{\beta_4(n)}\ln|z_n|.
\end{align}


\noindent The relation 
$	\beta_1(n)+\beta_2(n+3)+\beta_3(n+2)+\beta_4(n+1)=0$
obtained from the determining equations provides a natural guide for
constructing invariant quantities. \par \noindent Consider the logarithmic combination

\begin{equation}
\beta_1(n)s_1(n)+\beta_2(n+3)s_2(n+3)+\beta_3(n+2)s_3(n+2)+\beta_4(n+1)s_4(n+1)=	\ln|x_n|
	+\ln|y_{n+3}|
	+\ln|w_{n+2}|
	+\ln|z_{n+1}|.\label{lnn}
\end{equation}
Applying the corresponding prolonged symmetry generator to \eqref{lnn} gives
\begin{align}
	X\Big(
	\ln|x_n|
	+\ln|y_{n+3}|
	+\ln|w_{n+2}|
	+\ln|z_{n+1}|
	\Big)
	&=
	\beta_1(n)
	+\beta_2(n+3)
	+\beta_3(n+2)
	+\beta_4(n+1)
\nonumber	\\
	&=0.
\end{align}
Hence this combination is invariant under the admitted Lie symmetry. Therefore the invariant is directly determined by the symmetry
constraint obtained from the coefficient matching process. Exponentiating the logarithmic invariant produces
a multiplicative invariant, which is more convenient for the
subsequent reduction procedure. \par \noindent
By cyclic permutation of the variables, three additional logarithmic
invariants are obtained. Exponentiation of these logarithmic
expressions and taking the reciprocals produce the product invariants 
{\scriptsize \begin{align}
	\bar X_n=
	\frac1{x_ny_{n+3}w_{n+2}z_{n+1}},
	\bar Y_n=
	\frac1{x_{n+1}y_nw_{n+3}z_{n+2}},
	\bar W_n=
	\frac1{x_{n+2}y_{n+1}w_nz_{n+3}},
	\bar Z_n=
	\frac1{x_{n+3}y_{n+2}w_{n+1}z_n},
	\label{changezbar}
\end{align}}
used throughout the reduction procedure. It can be shown that 
\begin{align}
	\bar Z_{n+2}
	&=A_n\bar Y_n+B_n.
	\label{eq:barZ2}
\end{align}
Similarly,
\begin{align}
	\bar Y_{n+2}=E_n\bar Z_n+F_n,\;
	\bar W_{n+2}=G_n\bar X_n+H_n,\;
	\bar X_{n+2}=C_n\bar W_n+D_n.\label{eq:barX2}
\end{align}
Equations \eqref{eq:barZ2} and \eqref{eq:barX2} constitute a coupled
family of affine two-step recurrence relations for the reciprocal
invariants. 
\noindent Shifting \eqref{eq:barZ2} by two, we get
\begin{align}
	\bar Z_{n+4}=&A_{n+2}\bar Y_{n+2}+B_{n+2}\\
	=&A_{n+2}E_n\bar Z_n+(A_{n+2}F_n+B_{n+2}).
\end{align}
Thus
\begin{equation}
	\bar Z_{n+4}=\Theta_n^z\bar Z_n+\Lambda_n^z,
	\qquad \text{where}\qquad
	\Theta_n^z=A_{n+2}E_n,
	\qquad
	\Lambda_n^z=A_{n+2}F_n+B_{n+2}.
	\label{eq:Zfour}
\end{equation}
Similarly,
\begin{align}
	\bar X_{n+4}&=\Theta_n^x\bar X_n+\Lambda_n^x,
	& \text{where}\qquad\Theta_n^x&=C_{n+2}G_n,
	&\Lambda_n^x&=C_{n+2}H_n+D_{n+2},\label{eq:Xfour}\\
	\bar Y_{n+4}&=\Theta_n^y\bar Y_n+\Lambda_n^y,
	&\text{where}\qquad \Theta_n^y&=E_{n+2}A_n,
	&\Lambda_n^y&=E_{n+2}B_n+F_{n+2},\label{eq:Yfour}\\
	\bar W_{n+4}&=\Theta_n^w\bar W_n+\Lambda_n^w,
	&\text{where}\qquad\Theta_n^w&=G_{n+2}C_n,
	&\Lambda_n^w&=G_{n+2}D_n+H_{n+2}.\label{eq:Wfour}
\end{align}

\subsection{Solving the generic recurrence}

All four recurrences have the form
\begin{equation}
	U_{n+4}=\Theta_nU_n+\Lambda_n.
	\label{eq:genericU}
\end{equation}
Although the original system is fifth-order and therefore requires
five initial values, the invariant reduction produces four-step
recurrence relations for the reciprocal invariants. A simple iteration gives
\begin{equation}
	U_{4s+j}=U_j\prod_{t=0}^{s-1}\Theta_{4t+j}
	+\sum_{\ell=0}^{s-1}\left(\Lambda_{4\ell+j}
	\prod_{t=\ell+1}^{s-1}\Theta_{4t+j}\right),
	\label{eq:generic_solution}
\end{equation}
where an empty product is $1$.



\noindent From equation 
 \eqref{changezbar}, it follows that
\begin{align}
	x_{n+4}=&x_n\frac{\bar X_n}{\bar Z_{n+1}},\;
	y_{n+4}=y_n\frac{\bar Y_n}{\bar X_{n+1}},\;
	w_{n+4}=w_n\frac{\bar W_n}{\bar Y_{n+1}},\;
	z_{n+4}=z_n\frac{\bar Z_n}{\bar W_{n+1}}.\label{eq:zrecover}
\end{align}
Iterating the first expression in \eqref{eq:zrecover} gives the closed product formula
{\scriptsize
\begin{align}
	&x_{4n+j}=x_j\prod_{s=0}^{n-1}\frac{\bar{X}_{4s+j}}{\bar{Z}_{4s+j+1}}, j=0,1,2,3\nonumber\\
	&=x_j\prod\limits_{s=0}^{n-1}\frac{\bar{X}_{j}\left( \prod\limits_{k_1=0}^{s-1}\theta^x_{4k_1+j}\right)+\sum\limits_{l=0}^{s-1}\left(\Lambda_{4l+j}^x \prod\limits_{k_2=l+1}^{s-1}\theta_{4k_2+j}^x\right)}{\bar{Z}_{j+1}\left( \prod\limits_{k_1=0}^{s-1}\theta^z_{4k_1+j+1}\right)+\sum\limits_{l=0}^{s-1}\left(\Lambda_{4l+j+1}^z \prod\limits_{k_2=l+1}^{s-1}\theta_{4k_2+j+1}^z\right)}\nonumber\\&= \frac{x_{j+4}^n}{x_j^{n-1}}\prod\limits_{s=0}^{n-1}\frac{\left( \prod\limits_{k_1=0}^{s-1}\theta^x_{4k_1+j}\right)+{X}_{j}\sum\limits_{l=0}^{s-1}\left(\Lambda_{4l+j}^x \prod\limits_{k_2=l+1}^{s-1}\theta_{4k_2+j}^x\right)}{\left( \prod\limits_{k_1=0}^{s-1}\theta^z_{4k_1+j+1}\right)+{Z}_{j+1}\sum\limits_{l=0}^{s-1}\left(\Lambda_{4l+j+1}^z \prod\limits_{k_2=l+1}^{s-1}\theta_{4k_2+j+1}^z\right)}\nonumber
	\\&= \frac{x_{j+4}^n}{x_j^{n-1}}\;\;\;\mathclap{\prod\limits_{s=0}^{n-1}}\;\;\;\frac{\left(\;\;\;\mathclap{ \prod\limits_{k_1=0}^{s-1}}\;\;C_{4k_1+j+2}G_{4k_1+j}\right)+x_jy_{j+3}w_{j+2}z_{j+1}\sum\limits_{l=0}^{s-1}\left((C_{4l+j+2}H_{4l+j}+D_{4l+j+2}) \prod\limits_{k_2=l+1}^{s-1}C_{4k_2+j+2}G_{4k_2+j}\right)}{\left(\;\;\;\mathclap{ \prod\limits_{k_1=0}^{s-1}}\;\;A_{4k_1+j+3}E_{4k_1+j+1}\right)+x_{j+4}y_{j+3}w_{j+2}z_{j+1}\;\;\mathclap{\sum\limits_{l=0}^{s-1}}\;\;\left((A_{4l+j+3}F_{4l+j+1}+B_{4l+j+3})\;\;\;\mathclap{ \prod\limits_{k_2=l+1}^{s-1}}\;\;A_{4k_2+j+3}E_{4k_2+j+1}\right)}
	\label{eq:xclosed}
\end{align}}
for $j=0,1,2$; and {\scriptsize
\begin{align}
	&x_{4n+3}=x_3\prod_{s=0}^{n-1}\frac{\bar{X}_{4s+3}}{\bar{Z}_{4(s+1)}}\nonumber\\
	&=x_3\prod\limits_{s=0}^{n-1}\frac{\bar{X}_{3}\left( \prod\limits_{k_1=0}^{s-1}\theta^x_{4k_1+3}\right)+\sum\limits_{l=0}^{s-1}\left(\Lambda_{4l+3}^x \prod\limits_{k_2=l+1}^{s-1}\theta_{4k_2+3}^x\right)}{\bar{Z}_{0}\left( \prod\limits_{k_1=0}^{s}\theta^z_{4k_1}\right)+\sum\limits_{l=0}^{s}\left(\Lambda_{4l}^z \prod\limits_{k_2=l+1}^{s}\theta_{4k_2}^z\right)}\nonumber\\
	&=\frac{x_3y_2^nw_1^nz_0^n}{y_6^nw_5^nz_4^n}\;\;{\prod\limits_{s=0}^{n-1}}\;\;\frac{\left( \prod\limits_{k_1=0}^{s-1}C_{4k_1+5}G_{4k_1+3}\right)+x_3y_{6}w_{5}z_{4}\sum\limits_{l=0}^{s-1}\left((C_{4l+5}H_{4l+3}+D_{4l+5}) \prod\limits_{k_2=l+1}^{s-1}C_{4k_2+5}G_{4k_2+3}\right)}{\left( \prod\limits_{k_1=0}^{s}A_{4k_1+2}E_{4k_1}\right)+x_3y_2w_1z_0\sum\limits_{l=0}^{s}\left((A_{4l+2}F_{4l}+B_{4l+2}) \prod\limits_{k_2=l+1}^{s}(A_{4k_2+2}E_{4k_2})\right)}.\label{eq:x3closed}
	\end{align}}
Similarly,{\scriptsize
\begin{align}
	y_{4n+j}&= \frac{y_{j+4}^n}{y_j^{n-1}}\prod\limits_{s=0}^{n-1}\frac{\left( \prod\limits_{k_1=0}^{s-1}E_{4k_1+j+2}A_{4k_1+j}\right)+x_{j+1}y_{j}w_{j+3}z_{j+2}\sum\limits_{l=0}^{s-1}\left((E_{4l+j+2}B_{4l+j}+F_{4l+j+2}) \prod\limits_{k_2=l+1}^{s-1}E_{4k_2+j+2}A_{4k_2+j}\right)}{\left( \prod\limits_{k_1=0}^{s-1}C_{4k_1+j+3}G_{4k_1+j+1}\right)+x_{j+1}y_{j+4}w_{j+3}z_{j+2}\sum\limits_{l=0}^{s-1}\left((C_{4l+j+3}H_{4l+j+1}+D_{4l+j+3}) \prod\limits_{k_2=l+1}^{s-1}C_{4k_2+j+3}G_{4k_2+j+1}\right)}
	\label{eq:yclosed}\end{align}}
	for $j=0,1,2$ and {\scriptsize
		\begin{align}
			&y_{4n+3}=\frac{y_3x_0^nw_2^nz_1^n}{x_4^nw_6^nz_5^n}\prod\limits_{s=0}^{n-1}\frac{\left( \prod\limits_{k_1=0}^{s-1}E_{4k_1+5}A_{4k_1+3}\right)+x_4y_{3}w_{6}z_{5}\sum\limits_{l=0}^{s-1}\left((E_{4l+5}B_{4l+3}+F_{4l+5}) \prod\limits_{k_2=l+1}^{s-1}E_{4k_2+5}A_{4k_2+3}\right)}{\left( \prod\limits_{k_1=0}^{s}C_{4k_1+2}G_{4k_1}\right)+x_0y_3w_2z_1\sum\limits_{l=0}^{s}\left((C_{4l+2}H_{4l}+D_{4l+2}) \prod\limits_{k_2=l+1}^{s}(C_{4k_2+2}G_{4k_2})\right)};\label{eq:y3closed}
		\end{align}
		\begin{align}
			w_{4n+j}&= \frac{w_{j+4}^n}{w_j^{n-1}}\prod\limits_{s=0}^{n-1}\frac{\left( \prod\limits_{k_1=0}^{s-1}C_{4k_1+j}G_{4k_1+j+2}\right)+x_{j+2}y_{j+1}w_{j}z_{j+3}\sum\limits_{l=0}^{s-1}\left((D_{4l+j}G_{4l+j+2}+H_{4l+j+2}) \prod\limits_{k_2=l+1}^{s-1}C_{4k_2+j}G_{4k_2+j+2}\right)}{\left( \prod\limits_{k_1=0}^{s-1}E_{4k_1+j+3}A_{4k_1+j+1}\right)+x_{j+2}y_{j+1}w_{j+4}z_{j+3}\sum\limits_{l=0}^{s-1}\left((E_{4l+j+3}B_{4l+j+1}+F_{4l+j+3}) \prod\limits_{k_2=l+1}^{s-1}E_{4k_2+j+3}A_{4k_2+j+1}\right)}
			\label{eq:wclosed}\end{align}}
		for $j=0,1,2$ and {\scriptsize
		\begin{align}
			&w_{4n+3}=\frac{w_3x_1^ny_0^nz_2^n}{x_5^ny_4^nz_6^n}\prod\limits_{s=0}^{n-1}\frac{\left( \prod\limits_{k_1=0}^{s-1}C_{4k_1+3}G_{4k_1+5}\right)+x_5y_{4}w_{3}z_{6}\sum\limits_{l=0}^{s-1}\left((D_{4l+3}G_{4l+5}+H_{4l+5}) \prod\limits_{k_2=l+1}^{s-1}C_{4k_2+3}G_{4k_2+5}\right)}{\left( \prod\limits_{k_1=0}^{s}E_{4k_1+2}A_{4k_1}\right)+x_1y_0w_3z_2\sum\limits_{l=0}^{s}\left((E_{4l+2}B_{4l}+F_{4l+2}) \prod\limits_{k_2=l+1}^{s}(E_{4k_2+2}A_{4k_2})\right)}\label{eq:w3closed};
		\end{align}
\begin{align}
	z_{4n+j}&= \frac{z_{j+4}^n}{z_j^{n-1}}\prod\limits_{s=0}^{n-1}\frac{\left( \prod\limits_{k_1=0}^{s-1}A_{4k_1+j+2}E_{4k_1+j}\right)+x_{j+3}y_{j+2}w_{j+1}z_{j}\sum\limits_{l=0}^{s-1}\left((A_{4l+j+2}F_{4l+j}+B_{4l+j+2}) \prod\limits_{k_2=l+1}^{s-1}A_{4k_2+j+2}E_{4k_2+j}\right)}{\left( \prod\limits_{k_1=0}^{s-1}C_{4k_1+j+1}G_{4k_1+j+3}\right)+x_{j+3}y_{j+2}w_{j+1}z_{j+4}\sum\limits_{l=0}^{s-1}\left((G_{4l+j+3}D_{4l+j+1}+H_{4l+j+3}) \prod\limits_{k_2=l+1}^{s-1}C_{4k_2+j+1}G_{4k_2+j+3}\right)}
	\label{eq:zclosed}\end{align}}
for $j=0,1,2$ and {\scriptsize
\begin{align}
	&z_{4n+3}=\frac{z_3x_2^ny_1^nw_0^n}{x_6^ny_5^nw_4^n}\prod\limits_{s=0}^{n-1}\frac{\left( \prod\limits_{k_1=0}^{s-1}A_{4k_1+5}E_{4k_1+3}\right)+x_6y_{5}w_{4}z_{3}\sum\limits_{l=0}^{s-1}\left((A_{4l+5}F_{4l+3}+B_{4l+5}) \prod\limits_{k_2=l+1}^{s-1}A_{4k_2+5}E_{4k_2+3}\right)}{\left( \prod\limits_{k_1=0}^{s}C_{4k_1}G_{4k_1+2}\right)+x_2y_1w_0z_3\sum\limits_{l=0}^{s}\left((G_{4l+2}D_{4l}+H_{4l+2}) \prod\limits_{k_2=l+1}^{s}(C_{4k_2}G_{4k_2+2})\right)}.\label{eq:z3closed}
\end{align}}
The solution of \eqref{eq:forward_shifted_full}  is given by \eqref{eq:xclosed}--\eqref{eq:z3closed}. 
\section{Autonomous cases and recovery of some results in the literature}
\subsection{Autonomous cases}
Let $A_n=A,\,B_n=B, \,C_n=C,\, D_n=D,\,E_n=E,\,F_n=F,\, G_n=G$ and $H_n=H$. Equations \eqref{eq:xclosed}-\eqref{eq:x3closed} simplify to (we have used \eqref{eq:forward_shifted_full} to rewrite $x_5$, $x_6$, $y_6$, $w_5$ in terms of the initial conditions  viz $x_i,y_i, w_i, z_i, \,i=0,1,2,3,4,5$):
{\scriptsize
	\begin{align}
		x_{4n} =& \frac{x_{4}^n}{x_0^{n-1}}\prod\limits_{s=0}^{n-1}\frac{\left( CG\right)^s+x_0y_{3}w_{2}z_{1}(CH+D)\sum\limits_{l=0}^{s-1}\left( CG\right)^l}{\left( AE\right)^s+x_{4}y_{3}w_{2}z_{1}(AF+B)\sum\limits_{l=0}^{s-1}\left( AE\right)^l},\nonumber \\
		=&\frac{x_{4}^n}{x_0^{n-1}}\prod\limits_{s=1}^{n-1}\frac{\left( CG\right)^s+x_0y_{3}w_{2}z_{1}(CH+D)\sum\limits_{l=0}^{s-1}\left( CG\right)^l}{\left( AE\right)^s+x_{4}y_{3}w_{2}z_{1}(AF+B)\sum\limits_{l=0}^{s-1}\left( AE\right)^l}, 
		\label{eq:x0closedcst}\\
	x_{4n+1} =& \frac{x_{5}^n}{x_1^{n-1}}\prod\limits_{s=0}^{n-1}\frac{\left( CG\right)^s+x_1y_{4}w_{3}z_{2}(CH+D)\sum\limits_{l=0}^{s-1}\left( CG\right)^l}{\left( AE\right)^s+x_{5}y_{4}w_{3}z_{2}(AF+B)\sum\limits_{l=0}^{s-1}\left( AE\right)^l}, \nonumber\\
	=&\frac{x_{1}y_0^n}{y_4^n}\prod\limits_{s=0}^{n-1}\frac{\left( CG\right)^s+x_1y_{4}w_{3}z_{2}(CH+D)\sum\limits_{l=0}^{s-1}\left( CG\right)^l}{A\left( AE\right)^s+x_{1}y_{0}w_{3}z_{2}[B(AE)^s+(AF+B)\sum\limits_{l=0}^{s-1}\left( AE\right)^l]}
	\label{eq:x1closedcst}
\end{align}
\begin{align}
	x_{4n+2} =& \frac{x_{6}^n}{x_2^{n-1}}\prod\limits_{s=0}^{n-1}\frac{\left( CG\right)^s+x_2y_{5}w_{4}z_{3}(CH+D)\sum\limits_{l=0}^{s-1}\left( CG\right)^l}{\left( AE\right)^s+x_{6}y_{5}w_{4}z_{3}(AF+B)\sum\limits_{l=0}^{s-1}\left( AE\right)^l}, \nonumber\\
	=&\frac{x_{2}w_4^n}{w_0^{n}}\prod\limits_{s=0}^{n-1}\frac{C\left( CG\right)^s+x_2y_{1}w_{0}z_{3}[D(CG)^s+(CH+D)\sum\limits_{l=0}^{s-1}\left( CG\right)^l]}{A\left( AE\right)^s+x_{2}y_{1}w_{4}z_{3}[B(AE)^s+(AF+B)\sum\limits_{l=0}^{s-1}\left( AE\right)^l]}, 
	\label{eq:x2closedcst}\\
		x_{4n+3}=&\frac{x_3y_2^nw_1^nz_0^n}{y_6^nw_5^nz_4^n}\prod\limits_{s=0}^{n-1}\frac{\left(CG\right)^s+x_3y_{6}w_{5}z_{4}(CH+D)\sum\limits_{l=0}^{s-1}\left(CG\right)^l}{\left( AE\right)^{s+1}+x_3y_2w_1z_0(AF+B)\sum\limits_{l=0}^{s}\left( AE\right)^l}\nonumber\\
		=&\frac{x_3z_0^n}{z_4^n}\prod\limits_{s=0}^{n-1}\frac{C\left(CG\right)^s+x_3y_{2}w_{1}z_{4}[D(CG)^s+(CH+D)\sum\limits_{l=0}^{s-1}\left(CG\right)^l]}{\left( AE\right)^{s+1}+x_3y_2w_1z_0(AF+B)\sum\limits_{l=0}^{s}\left( AE\right)^l}.
		\label{eq:x3closedcst}
\end{align}}
Similarly, {\scriptsize
\begin{align}
	y_{4n}=& \frac{y_{4}^n}{y_0^{n-1}}\prod\limits_{s=0}^{n-1}\frac{\left(AE\right)^s+x_{1}y_{0}w_{3}z_{2}(BE+F)\sum\limits_{l=0}^{s-1}\left(AE\right)^l}{\left(CG\right)^s+x_{1}y_{4}w_{3}z_{2}(CH+D)\sum\limits_{l=0}^{s-1}\left(CG\right)^l}, 
	\label{eq:y0closedcst}\\
	y_{4n+1}=& \frac{y_{1}w_0^{n}}{w_4^n}\prod\limits_{s=0}^{n-1}\frac{\left(AE\right)^s+x_{2}y_{1}w_{4}z_{3}(BE+F)\sum\limits_{l=0}^{s-1}\left(AE\right)^l}{C\left(CG\right)^s+x_{2}y_{1}w_{0}z_{3}[D(CG)^s+(CH+D)\sum\limits_{l=0}^{s-1}\left(CG\right)^l]},	\label{eq:y1closedcst}\\
	y_{4n+2}=& \frac{y_{2}z_4^n}{z_0^{n}}\prod\limits_{s=0}^{n-1}\frac{E\left(AE\right)^s+x_{3}y_{2}w_{1}z_{0}[F(AE)^s+(BE+F)\sum\limits_{l=0}^{s-1}\left(AE\right)^l}{C\left(CG\right)^s+x_{3}y_{2}w_{1}z_{4}[D(CG)^s+(CH+D)\sum\limits_{l=0}^{s-1}\left(CG\right)^l]},\label{eq:y2closedcst}\\
	y_{4n+3}=&\frac{y_3x_0^n}{x_4^n}\prod\limits_{s=0}^{n-1}\frac{E\left( AE\right)^s+x_4y_{3}w_{2}z_{1}[F(AE)^s+(BE+F)\sum\limits_{l=0}^{s-1}\left( AE\right)^l]}{\left(CG\right)^{s+1}+x_0y_3w_2z_1(CH+D)\sum\limits_{l=0}^{s}\left(CG\right)^l};\label{eq:y3closedcst}\\
	w_{4n}&= \frac{w_{4}^n}{w_0^{n-1}}\prod\limits_{s=0}^{n-1}\frac{\left(CG\right)^s+x_{2}y_{1}w_{0}z_{3}(DG+H)\sum\limits_{l=0}^{s-1}\left(CG\right)^l}{\left(AE\right)^s+x_{2}y_{1}w_{4}z_{3}(BE+F)\sum\limits_{l=0}^{s-1}\left(AE\right)^l},\label{eq:w0closedcst}\\
	w_{4n+1}&= \frac{w_1z_0^n}{z_4^n}\prod\limits_{s=0}^{n-1}\frac{\left(CG\right)^s+x_{3}y_{2}w_{1}z_{4}(DG+H)\sum\limits_{l=0}^{s-1}\left(CG\right)^l}{E\left(AE\right)^s+x_{3}y_{2}w_{1}z_{0}[F(AE)^s+(BE+F)\sum\limits_{l=0}^{s-1}\left(AE\right)^l]}, \label{eq:w1closedcst}\\
	w_{4n+2}&= \frac{w_2x_{4}^n}{x_0^{n}}\prod\limits_{s=0}^{n-1}\frac{G\left(CG\right)^s+x_{0}y_{3}w_{2}z_{1}[H(CG)^s+(DG+H)\sum\limits_{l=0}^{s-1}\left(CG\right)^l]}{E\left(AE\right)^s+x_{4}y_{3}w_{2}z_{1}[F(AE)^s+(BE+F)\sum\limits_{l=0}^{s-1}\left(AE\right)^l]},\label{eq:w2closedcst}\\
	&w_{4n+3}=\frac{w_3y_0^n}{y_4^n}\prod\limits_{s=0}^{n-1}\frac{G\left(CG\right)^s+x_1y_{4}w_{3}z_{2}[H(CG)^s+(DG+H)\sum\limits_{l=0}^{s-1}\left(CG\right)^l]}{\left( AE\right)^{s+1}+x_1y_0w_3z_2(BE+F)\sum\limits_{l=0}^{s}\left(AE\right)^l}\label{eq:w3closedcst}
\end{align}
\begin{align}
	z_{4n}&= \frac{z_{4}^n}{z_0^{n-1}}\prod\limits_{s=0}^{n-1}\frac{\left(AE\right)^s+x_{3}y_{2}w_{1}z_{0}(AF+B)\sum\limits_{l=0}^{s-1}\left(AE\right)^l}{\left(CG\right)^s+x_{3}y_{2}w_{1}z_{4}(DG+H)\sum\limits_{l=0}^{s-1}\left(CG\right)^l},
	\label{eq:z0closedcst}\\
	z_{4n+1}&= \frac{z_1x_0^n}{x_4^{n}}\prod\limits_{s=0}^{n-1}\frac{\left(AE\right)^s+x_{4}y_{3}w_{2}z_{1}(AF+B)\sum\limits_{l=0}^{s-1}\left(AE\right)^l}{G\left(CG\right)^s+x_{0}y_{3}w_{2}z_{1}[H(CG)^s+(DG+H)\sum\limits_{l=0}^{s-1}\left(CG\right)^l]},	\label{eq:z1closedcst}\\
	z_{4n+2}&= \frac{z_2y_{4}^n}{y_0^{n}}\prod\limits_{s=0}^{n-1}\frac{A\left(AE\right)^s+x_{1}y_{0}w_{3}z_{2}[B(AE)^s+(AF+B)\sum\limits_{l=0}^{s-1}\left(AE\right)^l]}{G\left(CG\right)^s+x_{1}y_{4}w_{3}z_{2}[H(CG)^s+(DG+H)\sum\limits_{l=0}^{s-1}\left(CG\right)^l]},\label{eq:z2closedcst}\\
	z_{4n+3}=&\frac{z_3w_0^n}{w_4^n}\prod\limits_{s=0}^{n-1}\frac{A\left(AE\right)^s+x_2y_{1}w_{4}z_{3}[B(AE)^s+(AF+B)\sum\limits_{l=0}^{s-1}\left(AE\right)^l]}{\left(CG\right)^{s+1}+x_2y_1w_0z_3(DG+H)\sum\limits_{l=0}^{s}\left(CG\right)^l}.\label{eq:z3closedcst}
\end{align}}
The solutions obtained above are expressed in terms of indices shifted forward by four units. To recover the solutions of the original system considered in \eqref{eq:main_system}, it is sufficient to reverse these four index shifts. 
In other words, the solutions of  \eqref{eq:main_system}, for constant coefficients, are then given by
{\scriptsize	\begin{align}
		x_{4n-4}=&\frac{x_{0}^n}{x_{-4}^{n-1}}\prod\limits_{s=1}^{n-1}\frac{\left( CG\right)^s+x_{-4}y_{-1}w_{-2}z_{-3}(CH+D)\sum\limits_{l=0}^{s-1}\left( CG\right)^l}{\left( AE\right)^s+x_{0}y_{-1}w_{-2}z_{-3}(AF+B)\sum\limits_{l=0}^{s-1}\left( AE\right)^l}, 
		\label{eq:x0closedcstback}\\
		x_{4n-3}=&\frac{x_{-3}y_{-4}^n}{y_0^n}\prod\limits_{s=0}^{n-1}\frac{\left( CG\right)^s+x_{-3}y_{0}w_{-1}z_{-2}(CH+D)\sum\limits_{l=0}^{s-1}\left( CG\right)^l}{A\left( AE\right)^s+x_{-3}y_{-4}w_{-1}z_{-2}[B(AE)^s+(AF+B)\sum\limits_{l=0}^{s-1}\left( AE\right)^l]},
		\label{eq:x1closedcstback}\\
		x_{4n-2}=&\frac{x_{-2}w_0^n}{w_{-4}^{n}}\prod\limits_{s=0}^{n-1}\frac{C\left( CG\right)^s+x_{-2}y_{-3}w_{-4}z_{-1}[D(CG)^s+(CH+D)\sum\limits_{l=0}^{s-1}\left( CG\right)^l]}{A\left( AE\right)^s+x_{-2}y_{-3}w_{0}z_{-1}[B(AE)^s+(AF+B)\sum\limits_{l=0}^{s-1}\left( AE\right)^l]}, 
		\label{eq:x2closedcstback}\\
				x_{4n-1}=&\frac{x_{-1}z_{-4}^n}{z_0^n}\prod\limits_{s=0}^{n-1}\frac{C\left(CG\right)^s+x_{-1}y_{-2}w_{-3}z_{0}[D(CG)^s+(CH+D)\sum\limits_{l=0}^{s-1}\left(CG\right)^l]}{\left( AE\right)^{s+1}+x_{-1}y_{-2}w_{-3}z_{-4}(AF+B)\sum\limits_{l=0}^{s}\left( AE\right)^l};
				\label{eq:x3closedcstback}
		\end{align}
		\begin{align}
			y_{4n-4}=& \frac{y_{0}^n}{y_{-4}^{n-1}}\prod\limits_{s=0}^{n-1}\frac{\left(AE\right)^s+x_{-3}y_{-4}w_{-1}z_{-2}(BE+F)\sum\limits_{l=0}^{s-1}\left(AE\right)^l}{\left(CG\right)^s+x_{-3}y_{0}w_{-1}z_{-2}(CH+D)\sum\limits_{l=0}^{s-1}\left(CG\right)^l}, 
			\label{eq:y0closedcstback}\\
			y_{4n-3}=& \frac{y_{-3}w_{-4}^{n}}{w_0^n}\prod\limits_{s=0}^{n-1}\frac{\left(AE\right)^s+x_{-2}y_{-3}w_{0}z_{-1}(BE+F)\sum\limits_{l=0}^{s-1}\left(AE\right)^l}{C\left(CG\right)^s+x_{-2}y_{-3}w_{-4}z_{-1}[D(CG)^s+(CH+D)\sum\limits_{l=0}^{s-1}\left(CG\right)^l]},	\label{eq:y1closedcstback}\\
			y_{4n-2}=& \frac{y_{-2}z_0^n}{z_{-4}^{n}}\prod\limits_{s=0}^{n-1}\frac{E\left(AE\right)^s+x_{-1}y_{-2}w_{-3}z_{-4}[F(AE)^s+(BE+F)\sum\limits_{l=0}^{s-1}\left(AE\right)^l}{C\left(CG\right)^s+x_{-1}y_{-2}w_{-3}z_{0}[D(CG)^s+(CH+D)\sum\limits_{l=0}^{s-1}\left(CG\right)^l]},\label{eq:y2closedcstback}\\
			y_{4n-1}=&\frac{y_{-1}x_{-4}^n}{x_0^n}\prod\limits_{s=0}^{n-1}\frac{E\left( AE\right)^s+x_0y_{-1}w_{-4}z_{-3}[F(AE)^s+(BE+F)\sum\limits_{l=0}^{s-1}\left( AE\right)^l]}{\left(CG\right)^{s+1}+x_{-4}y_{-1}w_{-2}z_{-3}(CH+D)\sum\limits_{l=0}^{s}\left(CG\right)^l};\label{eq:y3closedcstback}
		\end{align}
		\begin{align}
			w_{4n-4}&= \frac{w_{0}^n}{w_{-4}^{n-1}}\prod\limits_{s=0}^{n-1}\frac{\left(CG\right)^s+x_{-2}y_{-3}w_{-4}z_{-1}(DG+H)\sum\limits_{l=0}^{s-1}\left(CG\right)^l}{\left(AE\right)^s+x_{-2}y_{-3}w_{0}z_{-1}(BE+F)\sum\limits_{l=0}^{s-1}\left(AE\right)^l},\label{eq:w0closedcstback}\\
			w_{4n-3}&= \frac{w_{-3}z_{-4}^n}{z_0^n}\prod\limits_{s=0}^{n-1}\frac{\left(CG\right)^s+x_{-1}y_{-2}w_{-3}z_{0}(DG+H)\sum\limits_{l=0}^{s-1}\left(CG\right)^l}{E\left(AE\right)^s+x_{-1}y_{-2}w_{-3}z_{-4}[F(AE)^s+(BE+F)\sum\limits_{l=0}^{s-1}\left(AE\right)^l]}, \label{eq:w1closedcstback}\\
			w_{4n-2}&= \frac{w_{-2}x_{0}^n}{x_{-4}^{n}}\prod\limits_{s=0}^{n-1}\frac{G\left(CG\right)^s+[H(CG)^s+(DG+H)\sum\limits_{l=0}^{s-1}\left(CG\right)^l]x_{-4}y_{-1}w_{-2}z_{-3}}{E\left(AE\right)^s+x_{0}y_{-1}w_{-2}z_{-3}[F(AE)^s+(BE+F)\sum\limits_{l=0}^{s-1}\left(AE\right)^l]},\label{eq:w2closedcstback}\\
			w_{4n-1}=&\frac{w_{-1}y_{-4}^n}{y_0^n}\prod\limits_{s=0}^{n-1}\frac{G\left(CG\right)^s+x_{-3}y_{0}w_{-1}z_{-2}[H(CG)^s+(DG+H)\sum\limits_{l=0}^{s-1}\left(CG\right)^l]}{\left( AE\right)^{s+1}+x_{-3}y_{-4}w_{-1}z_{-2}(BE+F)\sum\limits_{l=0}^{s}\left(AE\right)^l};\label{eq:w3closedcstback}
		\end{align}
		\begin{align}
			z_{4n-4}&= \frac{z_{0}^n}{z_{-4}^{n-1}}\prod\limits_{s=0}^{n-1}\frac{\left(AE\right)^s+x_{-1}y_{-2}w_{-3}z_{-4}(AF+B)\sum\limits_{l=0}^{s-1}\left(AE\right)^l}{\left(CG\right)^s+x_{-1}y_{-2}w_{-3}z_{0}(DG+H)\sum\limits_{l=0}^{s-1}\left(CG\right)^l},
			\label{eq:z0closedcstback}\\
			z_{4n-3}&= \frac{z_{-3}x_{-4}^n}{x_0^{n}}\prod\limits_{s=0}^{n-1}\frac{\left(AE\right)^s+x_{0}y_{-1}w_{-2}z_{-3}(AF+B)\sum\limits_{l=0}^{s-1}\left(AE\right)^l}{G\left(CG\right)^s+x_{-4}y_{-1}w_{-2}z_{-3}[H(CG)^s+(DG+H)\sum\limits_{l=0}^{s-1}\left(CG\right)^l]},	\label{eq:z1closedcstback}\\
			z_{4n-2}&= \frac{z_{-2}y_{0}^n}{y_{-4}^{n}}\prod\limits_{s=0}^{n-1}\frac{A\left(AE\right)^s+x_{-3}y_{-4}w_{-1}z_{-2}[B(AE)^s+(AF+B)\sum\limits_{l=0}^{s-1}\left(AE\right)^l]}{G\left(CG\right)^s+x_{-3}y_{0}w_{-1}z_{-2}[H(CG)^s+(DG+H)\sum\limits_{l=0}^{s-1}\left(CG\right)^l]},\label{eq:z2closedcstback}\\
			z_{4n-1}=&\frac{z_{-1}w_{-4}^n}{w_0^n}\prod\limits_{s=0}^{n-1}\frac{A\left(AE\right)^s+x_{-2}y_{-3}w_{0}z_{-1}[B(AE)^s+(AF+B)\sum\limits_{l=0}^{s-1}\left(AE\right)^l]}{\left(CG\right)^{s+1}+x_{-2}y_{-3}w_{-3}z_{-1}(DG+H)\sum\limits_{l=0}^{s}\left(CG\right)^l}.\label{eq:z3closedcstback}
		\end{align}}
For the sake of comparison with the results in \cite{ElsayedAlofi2023}, we made the following substitutions:
 %
$n\mapsto n+2$ in
$x_{4n-4},y_{4n-4},w_{4n-4},z_{4n-4}$, and
$n\mapsto n+1$ in each of
$x_{4n-3},y_{4n-3},w_{4n-3},z_{4n-3}$,
$x_{4n-2},y_{4n-2},w_{4n-2},z_{4n-2}$, and
$x_{4n-1},y_{4n-1},w_{4n-1},z_{4n-1}$.
After reindexing the running indices to start at \(n=0\), the solutions obtained take the following forms: 
{\scriptsize
	\begin{align}
		x_{4n+4}=&\frac{x_{0}^{n+2}}{x_{-4}^{n+1}}\prod\limits_{s=0}^{n}\frac{\left( CG\right)^{s+1}+x_{-4}y_{-1}w_{-2}z_{-3}(CH+D)\sum\limits_{l=0}^{s}\left( CG\right)^l}{\left( AE\right)^{s+1}+x_{0}y_{-1}w_{-2}z_{-3}(AF+B)\sum\limits_{l=0}^{s}\left( AE\right)^l}, 
		\label{eq:x0closedcstbackR}\\
		x_{4n+1}=&\frac{x_{-3}y_{-4}^{n+1}}{y_0^{n+1}}\frac{\prod\limits_{s=0}^{n-1}[\left( CG\right)^{s+1}+x_{-3}y_{0}w_{-1}z_{-2}(CH+D)\sum\limits_{l=0}^{s}\left( CG\right)^l]}{\prod\limits_{s=0}^{n}[A\left( AE\right)^s+x_{-3}y_{-4}w_{-1}z_{-2}[B(AE)^s+(AF+B)\sum\limits_{l=0}^{s-1}\left( AE\right)^l]]}
		\label{eq:x1closedcstbackR}\\
		x_{4n+2}=&\frac{x_{-2}w_0^{n+1}}{w_{-4}^{n+1}}\prod\limits_{s=0}^{n}\frac{C\left( CG\right)^s+x_{-2}y_{-3}w_{-4}z_{-1}[D(CG)^s+(CH+D)\sum\limits_{l=0}^{s-1}\left( CG\right)^l]}{A\left( AE\right)^s+x_{-2}y_{-3}w_{0}z_{-1}[B(AE)^s+(AF+B)\sum\limits_{l=0}^{s-1}\left( AE\right)^l]}, 
		\label{eq:x2closedcstbackR}\\
				x_{4n+3}=&\frac{x_{-1}z_{-4}^{n+1}}{z_0^{n+1}}\prod\limits_{s=0}^{n }\frac{C\left(CG\right)^s+x_{-1}y_{-2}w_{-3}z_{0}[D(CG)^s+(CH+D)\sum\limits_{l=0}^{s-1}\left(CG\right)^l]}{\left( AE\right)^{s+1}+x_{-1}y_{-2}w_{-3}z_{-4}(AF+B)\sum\limits_{l=0}^{s}\left( AE\right)^l}
				\label{eq:x3closedcstbackR}
		\end{align}
		\begin{align}
			y_{4n+4}=& \frac{y_{0}^{n+2}}{y_{-4}^{n+1}}\prod\limits_{s=0}^{n}\frac{\left(AE\right)^{s+1}+x_{-3}y_{-4}w_{-1}z_{-2}(BE+F)\sum\limits_{l=0}^{s}\left(AE\right)^l}{\left(CG\right)^{s+1}+x_{-3}y_{0}w_{-1}z_{-2}(CH+D)\sum\limits_{l=0}^{s}\left(CG\right)^l}, 
			\label{eq:y0closedcstbackR}\\
			y_{4n+1}=& \frac{y_{-3}w_{-4}^{n+1}}{w_0^{n+1}}\frac{\prod\limits_{s=0}^{n-1}[\left(AE\right)^{s+1}+x_{-2}y_{-3}w_{0}z_{-1}(BE+F)\sum\limits_{l=0}^{s}\left(AE\right)^l]}{\prod\limits_{s=0}^{n}[C\left(CG\right)^s+x_{-2}y_{-3}w_{-4}z_{-1}[D(CG)^s+(CH+D)\sum\limits_{l=0}^{s-1}\left(CG\right)^l]]},	\label{eq:y1closedcstbackR}\end{align}\begin{align}
			y_{4n+2}=& \frac{y_{-2}z_0^{n+1}}{z_{-4}^{n+1}}\prod\limits_{s=0}^{n}\frac{E\left(AE\right)^s+x_{-1}y_{-2}w_{-3}z_{-4}[F(AE)^s+(BE+F)\sum\limits_{l=0}^{s-1}\left(AE\right)^l}{C\left(CG\right)^s+x_{-1}y_{-2}w_{-3}z_{0}[D(CG)^s+(CH+D)\sum\limits_{l=0}^{s-1}\left(CG\right)^l]},\label{eq:y2closedcstbackR}\\
			y_{4n+3}=&\frac{y_{-1}x_{-4}^{n+1}}{x_0^{n+1}}\prod\limits_{s=0}^{n}\frac{E\left( AE\right)^s+x_0y_{-1}w_{-4}z_{-3}[F(AE)^s+(BE+F)\sum\limits_{l=0}^{s-1}\left( AE\right)^l]}{\left(CG\right)^{s+1}+x_{-4}y_{-1}w_{-2}z_{-3}(CH+D)\sum\limits_{l=0}^{s}\left(CG\right)^l};\label{eq:y3closedcstbackR}
		\end{align}
\begin{align}
w_{4n+4}&= \frac{w_{0}^{n+2}}{w_{-4}^{n+1}}\prod\limits_{s=0}^{n}\frac{\left(CG\right)^{s+1}+x_{-2}y_{-3}w_{-4}z_{-1}(DG+H)\sum\limits_{l=0}^{s}\left(CG\right)^l}{\left(AE\right)^{s+1}+x_{-2}y_{-3}w_{0}z_{-1}(BE+F)\sum\limits_{l=0}^{s}\left(AE\right)^l},\label{eq:w0closedcstbackR}\\
w_{4n+1}&= \frac{w_{-3}z_{-4}^{n+1}}{z_0^{n+1}}\frac{\prod\limits_{s=0}^{n-1}[\left(CG\right)^{s+1}+x_{-1}y_{-2}w_{-3}z_{0}(DG+H)\sum\limits_{l=0}^{s}\left(CG\right)^l]}{\prod\limits_{s=0}^{n}[E\left(AE\right)^s+x_{-1}y_{-2}w_{-3}z_{-4}[F(AE)^s+(BE+F)\sum\limits_{l=0}^{s-1}\left(AE\right)^l]]}, \label{eq:w1closedcstbackR}\\
w_{4n+2}&= \frac{w_{-2}x_{0}^{n+1}}{x_{-4}^{n+1}}\prod\limits_{s=0}^{n}\frac{G\left(CG\right)^s+[H(CG)^s+(DG+H)\sum\limits_{l=0}^{s-1}\left(CG\right)^l]x_{-4}y_{-1}w_{-2}z_{-3}}{E\left(AE\right)^s+x_{0}y_{-1}w_{-2}z_{-3}[F(AE)^s+(BE+F)\sum\limits_{l=0}^{s-1}\left(AE\right)^l]},\label{eq:w2closedcstbackR}\\
			w_{4n+3}=&\frac{w_{-1}y_{-4}^{n+1}}{y_0^{n+1}}\prod\limits_{s=0}^{n}\frac{G\left(CG\right)^s+x_{-3}y_{0}w_{-1}z_{-2}[H(CG)^s+(DG+H)\sum\limits_{l=0}^{s-1}\left(CG\right)^l]}{\left( AE\right)^{s+1}+x_{-3}y_{-4}w_{-1}z_{-2}(BE+F)\sum\limits_{l=0}^{s}\left(AE\right)^l}\label{eq:w3closedcstbackR}
		\end{align}
		\begin{align}
			z_{4n+4}&= \frac{z_{0}^{n+2}}{z_{-4}^{n+1}}\prod\limits_{s=0}^{n}\frac{\left(AE\right)^{s+1}+x_{-1}y_{-2}w_{-3}z_{-4}(AF+B)\sum\limits_{l=0}^{s}\left(AE\right)^l}{\left(CG\right)^{s+1}+x_{-1}y_{-2}w_{-3}z_{0}(DG+H)\sum\limits_{l=0}^{s}\left(CG\right)^l},
			\label{eq:z0closedcstbackR}\\
			z_{4n+1}&= \frac{z_{-3}x_{-4}^{n+1}}{x_0^{n+1}}\frac{\prod\limits_{s=0}^{n-1}[\left(AE\right)^{s+1}+x_{0}y_{-1}w_{-2}z_{-3}(AF+B)\sum\limits_{l=0}^{s}\left(AE\right)^l]}{\prod\limits_{s=0}^{n}[G\left(CG\right)^s+x_{-4}y_{-1}w_{-2}z_{-3}[H(CG)^s+(DG+H)\sum\limits_{l=0}^{s-1}\left(CG\right)^l]]},	\label{eq:z1closedcstbackR}\\
			z_{4n+2}&= \frac{z_{-2}y_{0}^{n+1}}{y_{-4}^{n+1}}\prod\limits_{s=0}^{n}\frac{A\left(AE\right)^s+x_{-3}y_{-4}w_{-1}z_{-2}[B(AE)^s+(AF+B)\sum\limits_{l=0}^{s-1}\left(AE\right)^l]}{G\left(CG\right)^s+x_{-3}y_{0}w_{-1}z_{-2}[H(CG)^s+(DG+H)\sum\limits_{l=0}^{s-1}\left(CG\right)^l]},\label{eq:z2closedcstbackR}\\
			z_{4n+3}=&\frac{z_{-1}w_{-4}^{n+1}}{w_0^{n+1}}\prod\limits_{s=0}^{n}\frac{A\left(AE\right)^s+x_{-2}y_{-3}w_{0}z_{-1}[B(AE)^s+(AF+B)\sum\limits_{l=0}^{s-1}\left(AE\right)^l]}{\left(CG\right)^{s+1}+x_{-2}y_{-3}w_{-3}z_{-1}(DG+H)\sum\limits_{l=0}^{s}\left(CG\right)^l}.\label{eq:z3closedcstbackR}
		\end{align}}
\subsection{Autonomous sign cases and recovery of the solutions of the eight systems in \cite{ElsayedAlofi2023}}
The autonomous sign systems are recovered by taking $A,B,C,D,E,F,G,H\in\{+1,-1\}$. \\
\underline{\textit{System 1}}\\
System 1 is obtained by choosing
$	A=1,\; B=-1,\; C=1,\; D=-1,\;
	E=1,\; F=-1,\; G=1,\; H=-1.$
The solutions \eqref{eq:x0closedcstbackR}-\eqref{eq:z3closedcstbackR} yield Theorem 1 in \cite{ElsayedAlofi2023}.
\noindent \underline{\textit{System 2}}\\
System 2 is obtained by choosing
$	A=-1,\quad B=-1,\quad C=-1,\quad D=-1,\quad
	E=-1,\quad F=-1,\quad G=-1,\quad H=-1$. The solutions \eqref{eq:x0closedcstbackR}-\eqref{eq:z3closedcstbackR} yield Theorem 2 in \cite{ElsayedAlofi2023}.\\
\noindent \underline{\textit{System 3}}\\
System 3 is recovered from the general system by choosing
$
A=-1,\quad B=-1,\quad C=-1,\quad D=-1,\quad
E=-1,\quad F=-1,\quad G=1,\quad H=-1.
$
Under this particular choice of parameters, the general solutions
\eqref{eq:x0closedcstbackR}-\eqref{eq:z3closedcstbackR} reduce to the solutions stated in Theorem 4 of \cite{ElsayedAlofi2023}. We illustrate this correspondence using the expression for $x_{4n+4}$. Indeed, substituting the above parameter values into \eqref{eq:x0closedcstbackR} gives
\begin{align}
	x_{4n+4}
	&=\frac{x_0^{n+2}}{x_{-4}^{n+1}}
	\prod_{s=0}^{n}(-1)^{s+1}=\frac{x_0^{n+2}}{x_{-4}^{n+1}}
	(-1)^{\frac{n^2+3n+2}{2}}.
	\label{sys3}
\end{align}
To obtain the subsequences appearing in Theorem 4, we separate the even and odd values of the index $n$. More precisely, setting $n:=2n$ in \eqref{sys3} yields
\begin{align}
	x_{8n+4}
	=\frac{x_0^{2n+2}}{x_{-4}^{2n+1}}
	(-1)^{n+1},
\end{align}
whereas setting $n:=2n+1$ gives
\begin{align}
	x_{8n+8}
	=\frac{x_0^{2n+3}}{x_{-4}^{2n+2}}
	(-1)^{n+1}.
\end{align}
These are precisely the two corresponding formulas for the $x$-sequence obtained in Theorem 4 of \cite{ElsayedAlofi2023}.
\par \noindent
The remaining formulas in that theorem are recovered in exactly the same manner. 


\noindent \underline{\textit{System 4}}\\
System 4 is obtained by choosing
$	A=-1,\quad B=-1,\quad C=-1,\quad D=-1,\quad
E=-1,\quad F=-1,\quad G=-1,\quad H=1.$  The solutions \eqref{eq:x0closedcstbackR}-\eqref{eq:z3closedcstbackR} yield Theorem 5 in \cite{ElsayedAlofi2023}.

\noindent \underline{\textit{System 5}}\\
System 5 is obtained by choosing
$	A=-1,\quad B=-1,\quad C=-1,\quad D=-1,\quad
E=-1,\quad F=-1,\quad G=1,\quad H=1.$  The solutions \eqref{eq:x0closedcstbackR}-\eqref{eq:z3closedcstbackR} yield Theorem 6 in \cite{ElsayedAlofi2023}.
To illustrate the reduction, substitute these values into
\eqref{eq:x0closedcstbackR}. Since
$
CG=-1,\; AE=1,\; CH+D=-2,
\; AF+B=0,$
we obtain
\begin{align}
x_{4n+4}
&=\frac{x_0^{n+2}}{x_{-4}^{n+1}}
\prod_{s=0}^{n}
\left[(-1)^{s+1}
-2x_{-4}y_{-1}w_{-2}z_{-3}
\sum_{l=0}^{s}(-1)^l\right].
\label{sys5}
\end{align}
Because
\[
\sum_{l=0}^{s}(-1)^l=
\begin{cases}
1,&s\ \text{even},\\
0,&s\ \text{odd},
\end{cases}
\]
formula \eqref{sys5} separates into the even and odd subsequences
appearing in Theorem 6 of \cite{ElsayedAlofi2023}. The remaining
component formulas follow from
\eqref{eq:x1closedcstbackR}--\eqref{eq:z3closedcstbackR} in the same
manner.

\noindent \underline{\textit{System 6}}\\
System 6 is obtained by choosing
$	A=-1,\quad B=-1,\quad C=-1,\quad D=-1,\quad
E=1,\quad F=1,\quad G=-1,\quad H=1.$  The solutions \eqref{eq:x0closedcstbackR}-\eqref{eq:z3closedcstbackR} yield Theorem 7 in \cite{ElsayedAlofi2023}.
For example, substituting these values into
\eqref{eq:x0closedcstbackR}, and using
$CG=1,\; AE=-1,\; CH+D=-2,
\; AF+B=-2,$ gives
\begin{align}
x_{4n+4}
&=\frac{x_0^{n+2}}{x_{-4}^{n+1}}
\prod_{s=0}^{n}
\frac{1-2(s+1)x_{-4}y_{-1}w_{-2}z_{-3}}
{(-1)^{s+1}-2x_0y_{-1}w_{-2}z_{-3}
\sum_{l=0}^{s}(-1)^l}.
\label{sys6}
\end{align}
Separating \eqref{sys6} according to the parity of $s$ yields the
corresponding $x$-subsequences in Theorem 7 of
\cite{ElsayedAlofi2023}. The other formulas are recovered similarly
from the remaining general expressions.

\noindent \underline{\textit{System 7}}\\
System 7 is obtained by choosing
$A=-1,\quad B=-1,\quad C=-1,\quad D=-1,\quad
E=1,\quad F=-1,\quad G=1,\quad H=-1.$  The solutions \eqref{eq:x0closedcstbackR}-\eqref{eq:z3closedcstbackR} yield Theorem 8 in \cite{ElsayedAlofi2023}. \\
\noindent \underline{\textit{System 8}}\\
System 8 is obtained by choosing
$A=-1,\quad B=-1,\quad C=-1,\quad D=-1,\quad
E=1,\quad F=-1,\quad G=-1,\quad H=1.$  The solutions \eqref{eq:x0closedcstbackR}-\eqref{eq:z3closedcstbackR} yield Theorem 9 in \cite{ElsayedAlofi2023}.
As a representative calculation, substituting these values into
\eqref{eq:x0closedcstbackR} gives $CG=1,\; AE=-1,\; CH+D=-2,
\; AF+B=0,$
and therefore
\begin{align}
x_{4n+4}
&=\frac{x_0^{n+2}}{x_{-4}^{n+1}}
(-1)^{\frac{(n+1)(n+2)}{2}}
\prod_{s=0}^{n}
\left[1-2(s+1)x_{-4}y_{-1}w_{-2}z_{-3}\right].
\label{sys8}
\end{align}
The parity decomposition of \eqref{sys8} gives the corresponding
expressions in Theorem 9 of \cite{ElsayedAlofi2023}. The remaining
solutions follow by applying the same substitutions to
\eqref{eq:x1closedcstbackR}--\eqref{eq:z3closedcstbackR}.
\begin{theorem}\label{theo1}
Let $(x_n)_{n\geq0}$, $(y_n)_{n\geq0}$, $(w_n)_{n\geq0}$ and
$(z_n)_{n\geq0}$ be solutions of \eqref{eq:forward_shifted_full} for
some non-zero constants $A,B,C,D,E,F,G,H$. Suppose that the initial conditions $x_i,y_i,w_i,z_i$, $i=0,\ldots,4$, satisfy
$x_0=x_4,\; y_0=y_4,\; w_0=w_4,\; z_0=z_4,$
$x_1y_0w_3z_2
=x_2y_1w_0z_3
=x_0y_3w_2z_1
=x_3y_2w_1z_0
=\frac{1-A}{B},$\; 
$x_1y_0w_3z_2
=x_2y_1w_0z_3
=x_0y_3w_2z_1
=x_3y_2w_1z_0
=\frac{1-C}{D},$ 
and
$(D-F)x_3y_2w_1z_0=E-C,\;
(F-H)x_0y_3w_2z_1=G-E,\;
(H-B)x_1y_0w_3z_2=A-G.$ 
Assume also that
$CH+D=AF+B=BE+F=DG+H.$
Then every component of the solution has period dividing $4$.
Moreover, if the complete solution has neither period $1$ nor period
$2$, then its least period is $4$.
\end{theorem}

\begin{proof}
Set $K=x_1y_0w_3z_2=x_2y_1w_0z_3
=x_0y_3w_2z_1=x_3y_2w_1z_0.$
The hypotheses imply $A+BK=C+DK=E+FK=G+HK=1.$ 
They also give
$AE=CG$ {and} $BE+F=DG+H$. Indeed, writing $A=1-BK,\; C=1-DK,
\; E=1-FK,\; G=1-HK$,  the coefficient relation $CH+D=AF+B=BE+F=DG+H$ 
 yields $(B+F)K-BFK^2=(D+H)K-DHK^2$, which is equivalent to $AE=CG$. We now use the autonomous solution formulas. For example,
\eqref{eq:w0closedcst} gives
\begin{align}
w_{4n}
&=\frac{w_4^n}{w_0^{n-1}}
\prod_{s=0}^{n-1}
\frac{(CG)^s+K(DG+H)\sum\limits_{l=0}^{s-1}(CG)^l}
{(AE)^s+K(BE+F)\sum\limits_{l=0}^{s-1}(AE)^l}\\
&=w_0,
\end{align}
since $w_4=w_0$, $AE=CG$, and $BE+F=DG+H$.
Applying the same substitutions to
\eqref{eq:x0closedcst}--\eqref{eq:z3closedcst} yields
\[
x_{4n+j}=x_j,\qquad y_{4n+j}=y_j,
\qquad w_{4n+j}=w_j,
\qquad z_{4n+j}=z_j,
\]
for $j=0,1,2,3$. Therefore each component has period dividing $4$.
If the complete solution has neither period $1$ nor period $2$, its
least period is $4$.
\end{proof}

\noindent For a numerical illustration, we let $A=C=E=G=2,\; B=D=F=H=-1$,
$(x_0,x_1,x_2,x_3,x_4)
=\left(\frac65,\frac45,\frac32,\frac7{10},\frac65\right),$ $ 
(y_0,y_1,y_2,y_3,y_4)
=\left(\frac9{10},\frac{11}{10},\frac34,\frac{13}{10},\frac9{10}\right),
(w_0,w_1,w_2,w_3,w_4)
=\left(\frac{21}{20},\frac{17}{20},\frac54,\frac{19}{20},\frac{21}{20}\right),$ $
(z_0,z_1,z_2,z_3,z_4)
=\left(\frac{800}{357},\frac{20}{39},\frac{250}{171},
\frac{400}{693},\frac{800}{357}\right).$
These values satisfy the hypotheses of Theorem~\ref{theo1} and do not
repeat after one or two indices. 

\begin{figure}[H]
\centering
\includegraphics[width=0.96\linewidth]{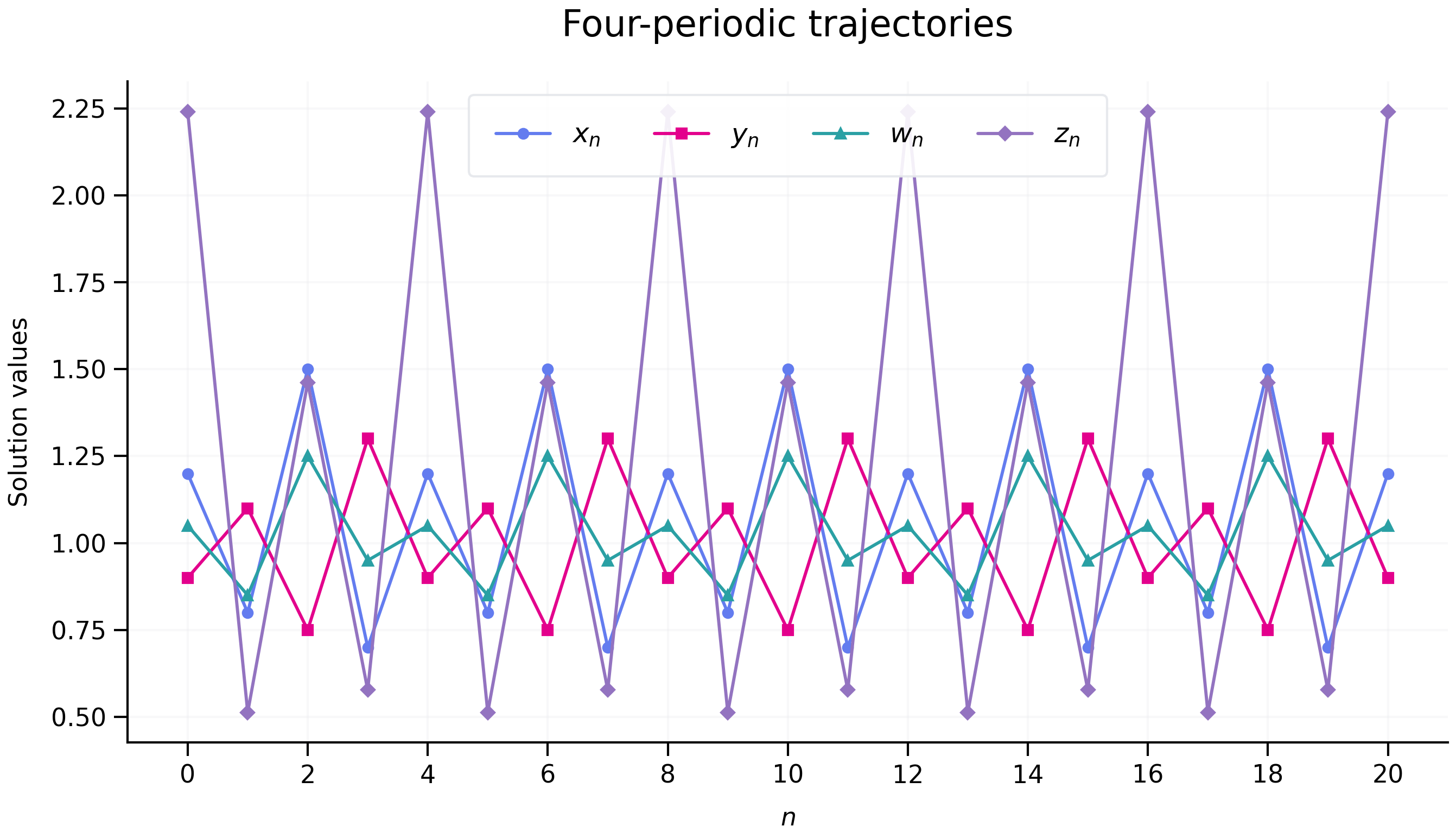}
\caption{Four-periodic trajectories of $x_n$, $y_n$, $w_n$ and $z_n$
satisfying the hypotheses of Theorem~\ref{theo1}.}
\label{fig:theorem21Overlay}
\end{figure}
\noindent Figure~\ref{fig:theorem21Overlay} illustrates the four-periodic
behaviour established in Theorem~\ref{theo1}. All four components
repeat after every four iterations, in agreement with the analytical
result.

\section{Conclusion}

In this paper, Lie symmetry analysis was applied to a four-dimensional fifth-order rational difference system with variable coefficients. The admitted symmetry generator was obtained from the determining equations, and the corresponding canonical coordinates and invariant quantities were constructed. These invariants reduced the original nonlinear system to a collection of recurrence relations involving reciprocal invariant variables. The reciprocal invariants were shown to satisfy affine two-step and four-step recurrence relations. These recurrence relations provided
the key reduction mechanism that enabled the recovery of the original variables through explicit product formulas. Solving the reduced
recurrences yielded explicit closed-form representations of the solutions. The developed framework extends previously studied autonomous sign systems by allowing arbitrary non-zero coefficient sequences. Furthermore, the eight autonomous sign systems of Elsayed and Alofi were
recovered as special cases of the general theory. The numerical illustration supports the four-periodic behaviour established analytically in Theorem~\ref{theo1}.
Future work may investigate periodic coefficient structures, attractivity and stability properties of the system of non-autonomous difference equations.

\section*{Conflict of interest} The authors declare that they have no conflict of interest.



\begin{thebibliography}{99}
		\bibitem{N1} Y. Akrour, M. Kara, N. Touafek and Y. Yazlik, {Solutions formulas for some general systems of difference equations}, \textit{Miskolc Mathematical Notes}, \textbf{22:2} (2021), 529--555.
	
	\bibitem{ElsayedAlofi2023}
	E.M. Elsayed and B.S. Alofi,
	The periodic nature and expression on solutions of some rational systems of
	difference equations,
	\emph{Alex. Eng. J.} \textbf{74} (2023), 269--283.
	
	
	
	
	\bibitem{MensahAnani2024}
	M. Folly-Gbetoula and Kwassi Anani, Method of Lie symmetry for analytical solutions, periodicity and attractivity of a family of tenth-order difference equations. \textit{Afrika Matematika} \textbf{36:122} (2025)
	
	
		\bibitem{M1} M. Folly-Gbetoula, Symmetry, reductions and exact solutions of the difference equation $u_{n+2}=au_n/(1+bu_nu_{n+1})$, {\it J. Differ. Equ. Appl.}, \textbf{23:6} (2017), 1017--1024.
	\bibitem{M2} M. Folly-Gbetoula, K. Mkhwanazi and  D. Nyirenda, On a study of a family  of higher order recurrence relations, {\it Math. Probl. Eng.} \textbf{2022} (2022), 6770105, 11 pages.
	\bibitem{M3} M. Folly-Gbetoula and A. H. Kara, Invariance analysis and reduction of discrete Painlev\'e equations, {\it J. Differ. Equ. Appl.}, \textbf{22:9} (2016), 1378--1388.
	
		\bibitem{GocenMensah2024}
	M. G\"o\c{c}en and M. Folly-Gbetoula,
	Form of the solutions of difference equations via Lie symmetry analysis and Fibonacci numbers,
	\emph{Quaest. Math.} \textbf{47}(2) (2024), 399--411.
	
	\bibitem{Hydon2014}
	P.E. Hydon,
	\emph{Difference Equations by Differential Equation Methods}, Cambridge
	University Press, Cambridge, 2014.
	
	\bibitem{NV} N. Joshi and P. Vassiliou, The existence of Lie Symmetries for First-Order Analytic Discrete Dynamical Systems, \textit{J. Math. Anal. Appl.}, \textbf{195} (1995), 872--887.
	
	\bibitem{Maeda1987}
	S. Maeda,
	The similarity method for difference equations,
	\emph{IMA Journal of Applied Mathematics} \textbf{38} (1987), 129--134.

\bibitem{M4}	K. Mkhwanazi  and  M. Folly-Gbetoula, Symmetries and Solvability of a Class of Higher Order Systems of Ordinary Difference Equations, {\it Symmetry} \textbf{ 14:1} (2022), 108.
\end{thebibliography}
\end{document}